\documentclass[12pt]{amsart}
\usepackage{amssymb,graphicx,array,xcolor}
\calclayout
\newcommand{\Sym}{\mathbb{S}} 
\newcommand{\R}{\mathbb{R}}

\def\SS{\mathcal S}

\newcommand\Mat{\mathbb{M}}
\renewcommand\P{{\mathcal P}}
\newcommand\Q{{\mathcal Q}}
\newcommand\RT{{\mathcal R\mathcal T}}
\newcommand\BDM{{\mathcal B\mathcal D\mathcal M}}
\newcommand\T{{\mathcal T}}

\newcommand\tr{\operatorname{tr}}

\newcommand\x{\times}
\renewcommand{\div}{\operatorname{div}} 
\newcommand{\grad}{\operatorname{grad}} 

\newcommand{\curl}{\operatorname{curl}} 
\newcommand{\as}{\operatorname{asym}} 
\newtheorem{thm}{Theorem}
\newtheorem{lemma}{Lemma}

\newtheorem{rem}{Remark} 
\def\<{\langle} 
\def\>{\rangle} 

\newcommand\norm[1]{\lVert#1\rVert}
\newcommand\0{\hphantom{0}}

\begin{document}

\title{Mixed finite elements for elasticity on quadrilateral meshes}
\thanks{The work of the first author was partially supported by
NSF grant DMS-1115291 and the Leverhulme Foundation.  The work of the second author
was partially supported by NSF grant DMS-0811052 and a 2009-2011
Sloan Foundation Fellowship.  This work was begun when the
authors were visitors of the Institute for Mathematics and its
Applications in 2010--2011 and completed while the first
author was visiting the University of Cambridge.
The authors would like to thank the anonymous referees for their suggestions.}

\author{Douglas N. Arnold, Gerard Awanou, and Weifeng Qiu}

\address{Department of Mathematics, University of Minnesota, Minneapolis,
Minnesota 55455}

\email{arnold@umn.edu}
\urladdr{http://www.ima.umn.edu/\~{}arnold}

\address{
Department of Mathematics, Statistics, and Computer Science (M/C 249), University of Illinois at Chicago,
Chicago, IL, 60607-7045}
\email{awanou@math.uic.edu}  
\urladdr{http://www.math.uic.edu/\~{}awanou}

\address{Department of Mathematics
Y6524 (Yellow Zone) 6/F Academic 1,
City University of Hong Kong,
Tat Chee Avenue,
Kowloon Tong,
Hong Kong}
\email{weifeqiu@cityu.edu.hk} 
\urladdr{http://www6.cityu.edu.hk/ma/people/profile/qiuf.htm}


\renewcommand{\subjclassname}{\textup{2000} Mathematics Subject Classification}

\maketitle

\begin{abstract}
We present stable mixed finite elements for planar linear
elasticity on general quadrilateral meshes.  The symmetry of
the stress tensor is imposed weakly and so there are three
primary variables, the stress tensor, the displacement vector
field, and the scalar rotation.  We develop and analyze a stable family of methods,
indexed by an integer $r \geq 2$ and with rate of convergence in the $L^2$ norm of
order $r$ for all the variables.
The methods use Raviart--Thomas elements for the
stress, piecewise tensor product polynomials for the displacement,
and piecewise polynomials for the rotation. 
We also present a simple
first order element, not belonging to this family.  It 
uses the lowest order BDM elements for the stress,
and piecewise constants for the displacement and rotation, and
achieves first order convergence for all three variables.

\end{abstract}
\maketitle

\section{Introduction}\label{s:intro}
In this paper we present \emph{mixed finite elements} for planar linear
elasticity which are stable for general quadrilateral meshes.
The mixed methods we consider are of the \emph{equilibrium} type in which
the approximate stress tensor $\sigma_h$ belongs to $H(\div)$ and
satisfies the equilibrium condition $\div\sigma_h = f$ exactly,
at least for loads $f$ which are piecewise polynomial of low
degree.  However, the methods are based on the mixed formulation
of elasticity with \emph{weakly imposed symmetry}, so that the
condition of balance of angular momentum, that is the symmetry
of the stress tensor, will be imposed only approximately, via a
Lagrange multiplier, which may be interpreted as the rotation.
Thus, we consider a formulation in which there are three primary
variables, the stress tensor, the displacement vector field,
and the scalar rotation.   See \eqref{elasSys} below.

For this formulation, we propose a family of stable triples of elements, one for each order $r\ge 2$.
The lowest order elements, $r=2$,  are illustrated
in Figure~\ref{elts1}.  For these we use the second lowest order quadrilateral
Raviart-Thomas elements for each row of the stress tensor,
discontinuous piecewise bilinear functions for each component of the displacement,
and discontinuous piecewise linear functions for the rotation. This method converges with second order
in the $L^2$ norm for all the variables. 
We also propose a simpler choice of elements,
illustrated in Figure~\ref{elts2}. It
uses the lowest order rectangular BDM elements for each row of the stress field and
piecewise constants for both the displacement and the rotation, and converges
with first order in the $L^2$ norm for all the variables. 

An important point is how the finite element shape functions are transformed from
a reference element to an actual quadrilateral element.  In order to achieve a stable
discretization we use different transformations for the stress, the displacement, and the rotation.
The displacement field is simply transformed by
composition with the inverse of the bilinear map from the reference element to the quadrilateral,
while the stress is mapped by the Piola transform (applied row-by-row).  The shape functions
for the rotation, in contrast, are not obtained by a transformation from the reference element,
but are simply the restriction of polynomials to the actual element.

Mixed finite elements for elasticity have many well-known advantages: robustness with
respect to material parameters, applicability to more general constitutive laws
such as viscoelasticity, etc.  Recently many mixed finite elements have been
developed, especially for the formulation in which the symmetry of the stress
tensor is imposed weakly (see the next section for a fuller discussion).
Stable elements have been developed for both triangles and rectangles.  The
latter apply easily to parallelograms as well.  However, up until now,
for the formulation with weakly imposed symmetry condition on the stress field,
there have been no stable mixed finite elements available for meshes
including general convex quadrilateral elements, even though such meshes
are preferred by many practitioners and implemented in many finite element
software systems.

Stable pairs of stress and displacement elements for equilibrium
mixed formulations of elasticity have been sought since the 1960s.
The first elements which were shown to be stable were proposed
in \cite{watwood-hartz} and analyzed in \cite{Johnson1978}.
These elements impose symmetry strongly, but they are composite
elements, in which the stress elements are piecewise linear with
respect to a subdivision into three triangles of each element of
the triangular mesh used for the piecewise linear displacements.
In \cite{Johnson1978} a quadrilateral version is analysed as well,
in which the stress uses a division into four triangular
microelements for each quadrilateral mesh element.  The first
stable elements with polynomial shape functions were not found for
triangular meshes until 2002 \cite{arnold-winther-02}, and then
developed for rectangular meshes in \cite{arnold-awanou-05}.
As far as we know, stable mixed finite elements with strong
symmetry and polynomial reference shape functions
have not yet been discovered for general quadrilateral
meshes.

Because of the difficulty in developing stable mixed methods
with strong symmetry, the idea of imposing symmetry weakly was
proposed already in 1965 \cite{fraeijsdv}.  The first stable
elements for this formulation were given in \cite{amara-thomas}
and \cite{peers}.  Since then numerous stable finite elements
with weak symmetry have been developed for simplicial meshes
\cite{stenberg86,stenberg87,stenberg88,farhloul-fortin}, especially
since the connection with the de~Rham complex and finite element
exterior calculus was made in \cite{AFW2006,AFW-acta}; besides
these papers, see \cite{BBF,CGG,GG2}.  Stable elements for
the mixed formulation with weak symmetry have been devised for
rectangular meshes as well \cite{morley,awanou}.  The element
which we develop in the next section of this paper are, to the
best of our knowledge, the first stable mixed finite elements
with weak symmetry for general quadrilateral meshes.
For a survey of mixed finite elements for elasticity through 2008,
we refer to \cite{Falk08b}.

In the following section we discuss mixed methods based on weakly
imposed symmetry in more detail, and recall the conditions
required for stable discretization and quasioptimal estimates.
In Section~\ref{s:construct}, we present a framework for the construction of stable
elements, based on two main ingredients: the connection between elasticity elements
and stable mixed finite elements for the Stokes equation and for
the Poisson equation, and the properties of various  transformations
of scalar, vector, and matrix fields.  Based on this framework, in Section~\ref{s:elts}
we define the finite elements described above and verify their stability.
In Section~\ref{s:rates}, we use the usual tools of mixed
methods to obtain improved rates of convergence in $L^2$. 
Finally, in Section~\ref{s:numer}, we illustrate the performance of the proposed
elements with numerical computations.

\section{Elasticity with weakly imposed symmetry and its discretization}\label{s:elas}
In this section we recall the weak formulation of the elasticity system
based on weak imposition of the symmetry of the stress tensor,
and its discretization by Galerkin's method.  We then summarize the
basic stability conditions and resulting error
estimate for such a method, and present a framework in which stable
subspaces can be constructed.

We write $\Mat$ and $\Sym$ for the spaces of $2\x2$ matrices and
symmetric matrices, respectively.  Let $\Omega$ be a bounded domain in $\R^2$
occupied by an elastic body.  The material properties are described,
at each point $x\in\Omega$, by the compliance tensor $A=A(x)$,
a linear operator $\Sym\to\Sym$ which is symmetric (with respect
to the Frobenius inner product) and positive definite.  We shall
assume that the compliance tensor is bounded and uniformly positive
definite on $\Omega$.  We shall also require an extension
of $A$ to an operator $\Mat\to\Mat$ which is still symmetric
and positive definite.  This can be obtained, for example, by defining
$A$ to act as a positive multiple of the identity on skew-symmetric
matrix fields.  In the case of a homogeneous and isotropic elastic
body,
$$
A \sigma = \frac{1}{2 \mu} \bigg( \sigma - \frac{\lambda}{2 \mu
  + 2 \lambda} \mathrm{tr} \ (\sigma) I \bigg),
\quad \sigma\in\Mat,
$$
where $I$ is the identity matrix and $\mu>0$ and $\lambda\ge0$ are the
Lam\'e constants.

Given a vector field
$f$ on $\Omega$ encoding the body forces, the equations of static
elasticity determine the stress $\sigma:\Omega\to\Sym$, 
and the displacement $u:\Omega\to\R^2$, satisfying the constitutive
equation $A\sigma = \epsilon(u)$, the equilibrium equation $\div\sigma = f$,
and boundary conditions, which, for simplicity, we take
to be $u=0$ on $\partial\Omega$.  Here $\epsilon(u)$ is the symmetric
part of the gradient of $u$ and the divergence operator $\div$ applies to
the matrix field $\sigma$ row-by-row.  Similarly below we shall define
$\curl w$ for  a vector field $w$ as the matrix field whose first row
is $\curl w_1$ and second row is $\curl w_2$, where $\curl q=(\partial_2 q,-\partial_1 q)$
for a scalar function $q$.

To derive
the weak formulation of elasticity which we shall use, we write
$\as\tau=\tau_{12}-\tau_{21}$ for the asymmetry of a matrix $\tau\in\Mat$
and introduce the rotation $p=\as(\grad u)/2$.  The constitutive equation
then becomes
$$
A\sigma = \grad u - \begin{pmatrix} 0 & p \\ -p & 0 \end{pmatrix}. 
$$
This equation, together with the equilibrium equation and the equation
$\as\sigma=0$ explicitly stating the symmetry of $\sigma$, form the
system of differential equations which we shall discretize.  For this we
shall use the weak formulation, which is to find
$(\sigma, u, p)\in H(\div,\Omega,\Mat) \x L^2(\Omega,\mathbb{R}^2)\x L^2(\Omega)$
such that
\begin{align}
\begin{split} \label{elasSys}
(A \sigma,\tau) + (u,\div \tau) + (p,\as \tau)& = 0, \quad \tau \in  H(\div,\Omega,\Mat),\\
(\div \sigma,v) & =(f,v),  \quad v \in L^2(\Omega,\mathbb{R}^2),\\
(\as \sigma,q) & = 0, \quad q \in L^2(\Omega,\mathbb{R}).
\end{split}
\end{align}

It is convenient to define the space
$$
Y = H(\div,\Omega,\Mat) \x L^2(\Omega,\mathbb{R}^2)\x L^2(\Omega)
$$
with the norms
$$
\norm{(\tau,v,q)}_Y=\norm{\tau}_{H(\div)}+\norm{v}_{L^2}+\norm{q}_{L^2},
\quad
\norm{(\tau,v,q)}_{L^2}=\norm{\tau}_{L^2}+\norm{v}_{L^2}+\norm{q}_{L^2},
$$
and to define $B:Y\x Y \to \R$, $F:Y\to \R$ by
\begin{gather}\label{defB}
B(\sigma,u,p;\tau,v,q) = (A \sigma,\tau) + (u,\div \tau) + (p,\as \tau)+
(\div \sigma,v)+ (\as \sigma,q),\\\label{defF}
F(\tau,v,p) = (f,v).
\end{gather}
Note that the bilinear form $B$ is  bounded with respect to the $Y$ norm,
with the bound depending only on the upper bound for the compliance tensor $A$.
In this notation, the weak formulation \eqref{elasSys}
takes the generic form: find $y=(\sigma,u,p)\in Y$ such that
$$
B(y,z) = F(z), \quad z\in Y.
$$
We approximate this by Galerkin's method using
finite element spaces $\Sigma_h\subset H(\div,\Omega,\Mat)$,
$V_h \subset L^2(\Omega,\mathbb{R}^2)$, and $Q_h\subset L^2(\Omega)$.
Setting $Y_h=\Sigma_h\x V_h\x Q_h$, the discrete solution
$y_h=(\sigma_h,u_h,p_h)\in Y_h$ is then defined by
$$
B(y_h,z) = F(z), \quad z\in Y_h.
$$

We now recall some basic stability and convergence results from
the theory of mixed methods.
For our problem, Brezzi's stability conditions \cite{brezzi} are:
\begin{itemize}
\item[(S1)] There exists a positive constant $c_1$ such that
$\norm{\tau}_{H(\div)}\le c_1(A\tau,\tau)^{1/2}$ whenever
$\tau\in\Sigma_h$ satisfies
$(\div\tau, v)=0$ for all $v\in V_h$
and $(\as\tau,q)=0$ for all $q\in Q_h$.
\item[(S2)] There exists a positive constant $c_2$ such
that for each $v\in V_h$ and $q\in Q_h$, there is a nonzero
$\tau\in\Sigma_h$ with
$$
(\div\tau, v)+(\as\tau,q)\ge
c_2\norm{\tau}_{H(\div)}(\norm{v}_{L^2}+\norm{q}_{L^2}).
$$
\end{itemize}
These conditions imply the inf-sup condition for the form
$B$:
\begin{itemize}
\item[(S0)] There exists a positive constant $c_0$
(depending on $c_1$ and $c_2$) such that
for each $y\in Y_h$ there is a nonzero $z\in Y_h$
with $B(y,z)\ge c_0 \norm{y}_Y\norm{z}_Y$.
\end{itemize}
This in turn implies that the Galerkin solution
$(\sigma_h,u_h,p_h)$ exists and is unique, and
that it satisfies a quasioptimal estimate with respect to
the norm in $Y$:
\begin{equation}\label{quasiopt}
\norm{\sigma-\sigma_h}_{H(\div)}+\norm{u-u_h}_{L^2}+\norm{p-p_h}_{L^2}
\le C\inf_{(\sigma,v,q)\in Y_h}(
\norm{\sigma-\tau}_{H(\div)}+\norm{u-v}_{L^2}+\norm{p-q}_{L^2}),
\end{equation}
with $C$ depending only on $c_1$, $c_2$, and an upper bound for $A$. In particular, the constant $C$ is independent 
of the Lam\'e parameter $\lambda$ if $c_1$ and $c_2$ are.

In the next section we study the construction of finite element spaces $\Sigma_h$,
$V_h$, and $Q_h$ satisfying (S1) and (S2).  First, however, we show that
these conditions hold at the continuous level, i.e., when $\Sigma_h$ is replaced
by $H(\div,\Omega,\Mat)$, $V_h$ by $L^2(\Omega,\mathbb{R}^2)$, and $Q_h$ by $L^2(\Omega)$,
and so that the weak problem is well-posed.
To prove the continuous analogue of (S2), we use the fact that for any
$q\in L^2(\Omega)$ there exists $w\in H^1(\Omega,\mathbb{R}^2)$ with $\div w=q$
and $\|w\|_{H^1}\le C\|q\|_{L^2}$.  For example, we may extend $q$ by zero
to a smoothly bounded domain and solve the Dirichlet problem for the Poisson
equation $\Delta u=q$ on that domain.  Then $w=\grad u|_\Omega$ has divergence
$q$ and satisfies the desired bound.

Now let $v\in L^2(\Omega,\R^2)$ and $q\in L^2(\Omega,\R)$. Then we can
choose $\eta \in H^1(\Omega,\mathbb{M})$ such that
$$
\div \eta = v, \quad \norm{\eta}_{H^1} \leq C \norm{v}_{L^2}.
$$
Similarly, we can  choose $w \in H^1(\Omega,\R^2)$ such that
\begin{align*}
\div w = q-\as \eta, \quad
\norm{w}_{H^1} &\leq C \norm{q-\as \eta}_{L^2}.
\end{align*}
If we then set $\tau= \eta - \curl w$, We have
\begin{align*}
\div \tau= \div \eta =v,\quad 
\as \tau = \as \eta+ \div w = q.
\end{align*}
Moreover
$$
\norm{\tau}_{H(\div)} \le \norm{\eta}_{H(\div)} +\norm{w}_{H^1}\le C(\norm{v}_{L^2}
+\norm{q}_{L^2}),
$$
for a constant $C>0$.  This suffices to establish (S2) at the continuous level.

The proof of (S1) at the continuous level is simple: the condition $(\div\tau, v)=0$ for all $v\in L^2(\Omega,\mathbb{R}^2)$
means that $\div\tau=0$, so $\norm{\tau}_{H(\div)}=\norm{\tau}_{L^2}$, which is bounded
by a constant multiple of $(A\tau,\tau)^{1/2}$, since the tensor $A$ is positive definite
for all $\mu>0$, $\lambda\ge 0$.  However, this
argument leads to a constant $c_1$ which is dependent not only on $\mu$, but also
on $\lambda$, and which tends to zero as $\lambda$ tends to infinity, since $A$ loses definiteness
in that limit.  The standard way to rectify this is
to use, instead of the positive definiteness of $A$, the estimate $(A\tau,\tau)\ge (2\mu)^{-1}\|\tau^D\|_{L^2}^2$
where $\tau^D$ is the \emph{deviatoric} or trace-free part $\tau$,
and to invoke the bound
$\|\tau\|_{L^2} \le c\|\tau^D\|_{L^2}$ for all $\tau\in H(\div,\Omega,\Mat)$
which are divergence-free and which satisfy the additional constraint
$\int_\Omega\tr\tau\,dx =0$.  This argument requires that the solution $\sigma$ satisfies
the constraint, for which it suffices to take the test function $\tau$ in \eqref{elasSys}
to be the constant matrix field everywhere equal to the identity.
In this way we may obtain well-posedness uniformly in $\lambda\ge0$.
For details, see, for instance, \cite{peers}, \cite{BBF}, or \cite[Prop.~9.1.1]{BBF-book}.

\section{Construction of stable elements}\label{s:construct}
In view of the preceding section, our goal is to construct
finite element spaces  $\Sigma_h\subset H(\div,\Omega,\Mat)$,
$V_h \subset L^2(\Omega,\mathbb{R}^2)$, and $Q_h\subset L^2(\Omega)$,
satisfying the stability conditions (S1) and (S2).  We shall
present such spaces in the next section.  In Section~\ref{ss:s2}, we consider
constructions that insure condition (S2), and in Section~\ref{ss:s1},
ones that insure (S1).

\subsection{The stability condition (S2)}\label{ss:s2}
In order to attain (S2), we exploit a connection between stable mixed finite
elements for elasticity with weak symmetry and stable mixed finite
elements for the Stokes and Poisson equations.  This connection,
which we recall in Theorem~\ref{s2}, was first observed in
\cite{farhloul-fortin} and has been elaborated and employed in, for example,  \cite{Falk08b,BBF,GG2}.
We note that it does not easily generalize to three dimensions.

A pair of spaces $W_h\subset H^1(\Omega,\R^2)$, $Q_h\subset L^2(\Omega)$,
is stable for the Stokes equations if it satisfies the appropriate inf-sup condition:
\begin{itemize}
\item[(S3)] There exists a positive constant $c_3$ such that
for each $q\in Q_h$ there is a nonzero $w\in W_h$
with $(\div w,q)\ge c_3 \norm{w}_{H^1} \norm{q}_{L^2}$.
\end{itemize}
Numerous stable Stokes pairs are known, and in Section~\ref{s:elts}
we shall choose from among them in order to fulfil (S3).

It is also useful to recall an equivalent form of (S3).
\begin{lemma} \label{eqS3}
 The inf-sup condition {\rm(S3)} holds for some positive constant
$c_3$ if and only if
for all $q\in Q_h$ there exists $w\in W_h$ such that
$P_{Q_h}\div w=q$ and $\norm{w}_{H^1}\le c_3^{-1}\norm{q}_{L^2}$,
where $P_{Q_h}:L^2(\Omega)\to Q_h$ is the $L^2$-projection.
\end{lemma}
\begin{proof}
Let $L_h=P_{Q_h}\div|_{W_h}:W_h\to Q_h$, and
let $L_h^*:Q_h\to W_h$ be its Hilbert space adjoint, where,
as norms on $W_h$ and $Q_h$
we use the $H^1$ and $L^2$ norms, respectively.  Note that
$$
\sup_{w\in W_h}\frac{(\div w,q)}{\norm{w}_{H^1}}
= \sup_{w\in W_h}\frac{(L_h w,q)_{Q_h}}{\norm{w}_{W_h}}
= \sup_{w\in W_h}\frac{( w,L_h^*q)_{W_h}}{\norm{w}_{W_h}}
= \norm{L_h^*q}_{W_h},
$$
so condition (S3) states that
$$
\norm{L_h^*q}_{W_h} \ge c_3\norm{q}_{Q_h}, \quad q\in Q_h,
$$
which is equivalent to stating that $L_h^*$ is an injective map of $Q_h$ onto a subspace of $W_h$
with inverse bounded by $c_3^{-1}$.  This in turn is equivalent to
the statement that $L_h$ is a surjective map of $W_h$ onto $Q_h$ and
admits a right-inverse bounded by $c_3^{-1}$, which is the desired
condition.\qed
\end{proof}

For the mixed Poisson equation, the inf-sup condition uses the
$H(\div)$ norm rather than the $H^1$ norm.  That is,
a pair of spaces $S_h\subset H(\div,\Omega,\R^2)$,
$U_h\subset L^2(\Omega)$ are required to satisfy the condition:
\begin{itemize}
\item[(S4)] There exists a positive constant $c_4$ such that
for each $q\in U_h$ there is a nonzero $w\in S_h$
with $(\div w,q)\ge c_4 \norm{w}_{H(\div)} \norm{q}_{L^2}$.
\end{itemize}
Again, there are numerous pairs of spaces known to satisfy (S4).
The next theorem gives the connection to mixed elasticity
elements.  It states that, if we choose a pair of spaces satisfying (S3) and another
satisfying (S4), and if the two choices satisfy the
compatibility condition  \eqref{compat} below, then we
obtain spaces satisfying (S2).
\begin{thm}\label{s2}
Suppose that $W_h\subset H^1(\Omega,\R^2)$ and $Q_h\subset L^2(\Omega)$
satisfy {\rm(S3)} and that $S_h\subset H(\div,\Omega,\R^2)$ and
$U_h\subset L^2(\Omega)$ satisfy {\rm(S4)}.  Suppose further
that
\begin{equation}\label{compat}
 \curl W_h\subset S_h\x S_h.
\end{equation}
Then
$\Sigma_h:= S_h\x S_h\subset H(\div,\Omega,\Mat)$
and $V_h:= U_h\x U_h\subset L^2(\Omega,\R^2)$ and $Q_h\subset L^2(\Omega)$ satisfy \rm{(S2)}.
\end{thm}
\begin{proof}
Let $v\in V_h$, $q\in Q_h$ be given.  Since $\Sigma_h= S_h\x S_h$ and $V_h= U_h\x U_h$,
(S4) implies that there exists $\eta\in\Sigma_h$ such that
$$
(\div\eta,v)=\norm{v}_{L^2}^2,\quad \norm{\eta}_{H(\div)} \le c_4^{-1}\norm{v}_{L^2}.
$$
Next we invoke (S3) with $q$ replaced by $q-P_{Q_h}(\as\eta)$. By Lemma \ref{eqS3}, there exists $w\in W_h$ such that
$$
P_{Q_h}(\div w)= q-P_{Q_h}(\as\eta),\quad
\norm{w}_{H^1} \le c_3^{-1}(\norm{q}_{L^2} + \norm{\eta}_{L^2}).
$$
Set
$$
\tau=\eta-\curl w\in\Sigma_h.
$$
Then
$$
(\div\tau,v) = (\div\eta,v) = \norm{v}_{L^2}^2.
$$
Also, since $\as(\curl w) = -\div w$,
\begin{align*}
 (\as\tau,q)&= (\as\eta,q) +(\div w,q) \\
 &= (P_{Q_h}(\as\eta),q)+(q-P_{Q_h}(\as\eta), q) = \norm{q}_{L^2}^2,
\end{align*}
and
$$
\norm{\tau}_{H(\div)} \le \norm{\eta}_{H(\div)} +\norm{w}_{H^1}\le C(\norm{v}_{L^2}
+\norm{q}_{L^2}),
$$
where $C$ depends only on $c_3$ and $c_4$.  This completes the verification of (S2).
\qed
\end{proof}

\subsection{The stability condition (S1)}\label{ss:s1}
The key to obtaining (S1)
will be the construction of the finite element spaces $\Sigma_h$ and $V_h$ from
shape function spaces $\hat\Sigma\subset H(\div,\hat K,\Mat)$ and
$\hat V\subset L^2(\hat K,\R^2)$ on a reference element $\hat K$, which
are transformed to a general element 
using appropriate transformations.  We define these transformations now
and summarize their main properties in Lemma~\ref{transint} below.
Based on these we establish (S1) in Theorem~\ref{s1}.

Suppose that $F_K:\hat K \to K$ is a diffeomorphism of bounded domains in the
plane.  (In the applications
in the next section, $\hat K$ will be the unit square and $F_K$ will be an invertible
bilinear map onto a convex quadrilateral $K$.)  A scalar- or vector-valued function
$\hat q$ on $\hat K$ transforms to a function $P_K^0\hat q$ on $K$ by composition:
$$
P_K^0 \hat q(x) = \hat q(\hat x),
$$
where $x=F_K(\hat x)$.
A different way to transform a scalar- or vector-valued function brings in the Jacobian determinant
$J_K=\det\grad F_K$:
$$
P_K^2\hat q(x) = \frac{1}{J_K(\hat x)}\hat q(\hat x).
$$
The notation refers to exterior calculus: $P_K^0$ corresponds to pull back by $F_K^{-1}$ if we
think of $\hat q$ as a differential $0$-form on $\hat K$, and $P_K^2$ corresponds to pull back as
a $2$-form.
A third way to transform a vector-valued function is to treat it as a $1$-form, i.e., to use
the Piola transform:
\begin{equation} \label{piola}
P_K^1 \hat{v}(x) = \frac{1}{J_K(\hat x)} [\grad F_K(\hat x)] \hat{v}(\hat x).
\end{equation}
We can also transform a matrix-valued function
on $\hat K$ to one on $K$ by applying the Piola transform to each row.  This transformation will
also be denoted by $P_K^1$.  We have the following fundamental identities.
\begin{lemma}
\begin{equation}\label{transprop}
\curl P_K^0\hat v  =P_K^1(\curl\hat v), \quad
\div P_K^1\hat \tau = P_K^2(\div\hat \tau),
\end{equation}
and
\begin{equation}\label{transint}
(P_K^2 \hat q, P_0 \hat v)_{L^2(K)} = (\hat q, \hat v)_{L^2(\hat K)} , \quad
 (\div P_K^1\hat \tau, P_K^0\hat v)_{L^2(K)} = (\div\hat\tau,\hat v)_{L^2(\hat K)}.
\end{equation}

\end{lemma}
\begin{proof}
The above relationships follow naturally
in exterior calculus, or can be verified by elementary vector calculus.  
\qed
\end{proof}

Now, let $\hat K\subset\R^2$ be a fixed reference element (e.g., the unit square),
and suppose that $\T_h$ is a partition
of $\Omega$ into finite elements such that for each $K\in\T_h$ there 
is given a diffeomorphism $F_K$ of $\hat K$ onto $K$. Suppose we are given a reference
shape function space $\hat V\subset L^2(\hat K,\R^2)$ and that the finite element
space $V_h$ is defined by
\begin{equation}\label{vdef}
V_h = \{ \, v \in L^2(\Omega,\R^2)\,:\, v|_K \in P_K^0 \hat{V},\  \forall K \in \mathcal{T}_h \, \}.
\end{equation}
Further assume given a reference shape function space $\hat\Sigma\subset H(\div,\hat K,\Mat)$
and suppose that the finite element space $\Sigma_h$ satisfies
\begin{equation}\label{sigdef}
\Sigma_h = \{\,\tau\in H(\div,\Omega,\Mat)\,:\, \tau|_K \in P_K^1 \hat\Sigma, \ \forall K\in \T_h \,\}.
\end{equation}
Finally, assume that the shape function spaces are related by the inclusion
\begin{equation}\label{incl}
\div\hat\Sigma \subset \hat V.
\end{equation}
These conditions imply (S1).
\begin{thm}\label{s1}
 If the shape function spaces $\hat V$ and $\hat\Sigma$ satisfy \eqref{incl}, and the finite
element spaces $V_h$ and $\Sigma_h$ are defined by \eqref{vdef} and \eqref{sigdef}, then {\rm(S1)} holds.
\end{thm}
\begin{proof}
 It is certainly sufficient to prove
that if $\tau\in \Sigma_h$ and $(\div\tau,v)=0$ for all $v\in V_h$, then $\div\tau=0$.
Indeed, this property implies (S1). 

Pick $K\in\T_h$ and set $\hat\tau=(P_K^1)^{-1}(\tau|_K)$,
$\hat v = \div\hat\tau$.  By \eqref{sigdef}, $\hat\tau\in \hat\Sigma$, and by \eqref{incl},
$\hat v\in \hat V$.  Define $v\in L^2(\Omega,\R^2)$ by $v|_K = (P_K^0)^{-1}\hat v$,
and $v\equiv0$ on $\Omega\setminus K$.  By \eqref{vdef}, $v\in V_h$,  and so, by assumption,
$(\div\tau,v)=0$. Using \eqref{transint},
$$
(\div\tau,v)=(\div\tau|_K,v|_K)_{L^2(K)} = (\div\hat\tau,\hat v)_{L^2(\hat K)} = \norm{\div\hat\tau}_{L^2(\hat K)}^2.
$$
Thus $\div\hat\tau=0$, and so, with \eqref{transprop} $\div\tau|_K = P_K^2(\div\hat \tau)=0$.  Since $K$ was arbitrary,
this shows that $\div\tau$ vanishes, as desired.
\qed
\end{proof}

\begin{rem}
 This argument leads to the constant $c_1$ in (S1) depending on both $\lambda$ and $\mu$.
 Just as for the continuous case discussed at the end of Section~2, a slightly more
 elaborate argument shows that $c_1$ can be taken independent of $\lambda$. 
 For this we need to choose the test function $\tau$ equal to the constant identity
 matrix in order to show that the $\sigma_h$ satisfies the constraint
 $\int\tr\sigma_h\,dx=0$.  Thus we have to check that the constant identity matrix
 field belongs to $\Sigma_h$.  From the definition \eqref{sigdef} this means
 checking that $(P^1_K)^{-1}I\in \hat \Sigma$, i.e., that
 $J_K(\hat x)[\grad F_K(\hat x)]^{-1}\in\hat\Sigma$.    Now $J_K(\hat x)[\grad F_K(\hat x)]^{-1}$
 is the transposed matrix of cofactors of the Jacobian matrix $\grad F_K(\hat x)$.  Since the components
 of $F_K(\hat x)$ are bilinear, the cofactors are linear polynomials.    Thus, as
 long as the reference space function space $\hat \Sigma$ contains the space $\P_1(\hat K,\Mat)$,
 then Theorem 2 results in (S1) holding with constant $c_1$ independent of $\lambda$,
 and the resulting mixed method will not
 exhibit locking for nearly incompressible materials.  This is the case for all of the choices of
 $\hat\Sigma$ we make below.
\end{rem}

\section{Stable elements for elasticity}\label{s:elts}
Theorems~\ref{s2} and \ref{s1} give strong guidance on the construction of stable
spaces $\Sigma_h$, $V_h$, and $Q_h$ for elasticity.  First, we require spaces $W_h$, $Q_h$,
$S_h$, $U_h$ which satisfy the hypotheses of Theorem~\ref{s2}, i.e., the first two
form a stable pair for the Stokes equations and the latter two a stable pair for the
mixed Poisson equation, and the compatibility condition \eqref{compat} is satisfied.
In order that the hypothesis of Theorem~\ref{s1} are also met, we will construct
these four spaces starting with shape functions on a reference element using
appropriate transformations.  Finally, we take $\Sigma_h=S_h\x S_h$, $V_h=U_h\x U_h$,
and $Q_h$ as our elements for the stress, displacement, and rotation.  Note
that the space $W_h$ (the Stokes velocity space) is only used for the analysis,
and does not enter the mixed method for elasticity.

We henceforth denote by $\hat K$ the unit square, and we assume that the partition
$\T_h$ of $\Omega$ consists of convex quadrilaterals $K$, and that each $F_K$
is a bilinear isomorphism of $\hat K$ onto $K$.
We assume that $\T_h$ is shape regular in the sense of \cite[p.~105]{Girault-Raviart}. To
define this, we consider for each convex quadrilateral the four triangles obtained by connecting three
of its vertices and let $\rho_K$ be the smallest of the diameters of the corresponding inscribed circles.
A sequence of quadrilateral meshes is shape regular if there is a constant
$\sigma$ such that $\operatorname{diam}(K)/\rho_K\le \sigma$ for
all the elements in the meshes.

\subsection{A first choice of elements}
\label{subsection_space1}
Let $\P_r$ denote the space of polynomials of degree at most $r$, and $\P_{r,s}$
the space of polynomials of degree at most $r$ in $x_1$ and $s$ in $x_2$.
We write $\Q_r$ for $\P_{r,r}$, and $\RT_r = \P_{r,r-1}\x\P_{r-1,r}$.  The
last space consists of the shape functions for the Raviart--Thomas space on a square.
For $K\subset\R^2$ we write $\P_r(K)$ for functions on $K$ obtained by restriction
of polynomials in $\P_r$, and use a similar notation for the other spaces.

For our first choice of elements,
the vector-valued finite element spaces $W_h$ and $S_h$ will be constructed
starting from reference shape function spaces:
$$
\hat W = \Q_2(\hat K)\x\Q_2(\hat K), \quad
\hat S = \RT_2(\hat K).
$$
These satisfy
\begin{equation}\label{curl}
\curl \hat W \subset \hat S\x \hat S.
\end{equation}
We then set
\begin{align}\label{defwh}
 W_h &= \{\, w\in H^1(\Omega,\R^2)\,:\, w|_K\in P_K^0 \hat W,\ \forall K\in\T_h\,\},
\\ \label{defsh}
 S_h &= \{\, s\in H(\div,\Omega,\R^2)\,:\, s|_K\in P_K^1 \hat S,\ \forall K\in\T_h\,\}.
\end{align}
Note that the transform $P_K^0$ is used to define $W_h$, but
the Piola transform $P_K^1$ is used in the definition of $S_h$.
Using \eqref{curl} and the first property in \eqref{transprop} of the transformations, we
see that the crucial compatibility condition \eqref{compat}
is 
satisfied.

The scalar-valued space $U_h$ is also defined starting with reference shape
functions.  We choose $\hat U=\Q_1$ and define
\begin{equation}\label{defuh}
 U_h = \{\, u\in L^2(\Omega)\,:\, u|_K\in P_K^0 \hat U,\ \forall K\in\T_h\,\}.
\end{equation}
In contrast, the scalar-valued space $Q_h$ is defined directly using polynomials on
the elements of $\T_h$ with no interelement continuity:
\begin{equation*}
 Q_h = \{\, q\in L^2(\Omega)\,:\, q|_K\in \P_1(K),\ \forall K\in\T_h\,\}.
\end{equation*}

Each of the spaces $W_h$, $Q_h$, $S_h$, $U_h$ has a standard set of degrees of freedom which enforce
the desired degree of continuity for the assembled spaces $W_h$, $Q_h$, $S_h$,
and $U_h$.  For $\hat W$ the degrees of freedom are the values of both components at the vertices
of the square, the integral of both components on the edges, and the integral
of both components over the square.  For $\hat S$ they are the averages
and first moments of the normal component on each edge and the interior moments weighted
by $\P_{0,1}\x\P_{1,0}$.  For $\hat Q$ and $\hat U$ all the degrees
of freedom are interior.
Figure~\ref{dofs1} illustrates the degrees of freedom for the four spaces, and
also includes an indication of how the shape functions transform to the
reference element for each space.  Note that the functions in $W_h$ are vector fields,
so each of the dots in the corresponding diagram represent two degrees of freedom.
\begin{figure}[htb]
\centerline{%
\begin{tabular}{cccc}
 $W_h$ ($P_K^0$)  & $Q_h$ (unmapped) & $S_h$ ($P_K^1$) & $U_h$ ($P_K^0$) 
\\
\includegraphics[width=30mm]{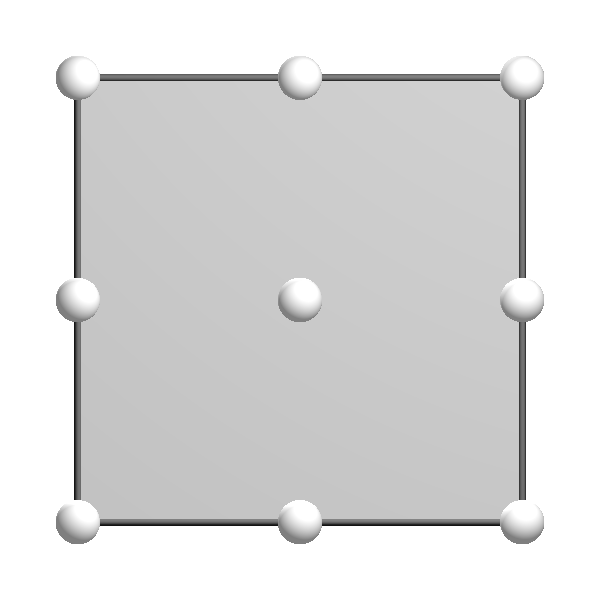} &
\includegraphics[width=30mm]{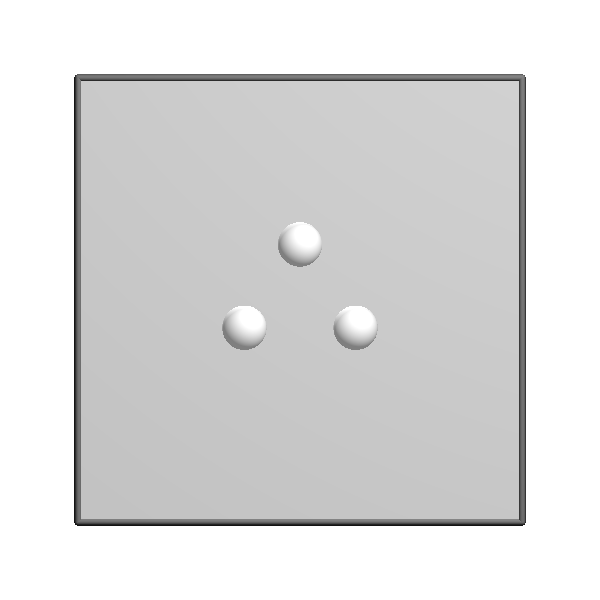} &
\includegraphics[width=30mm]{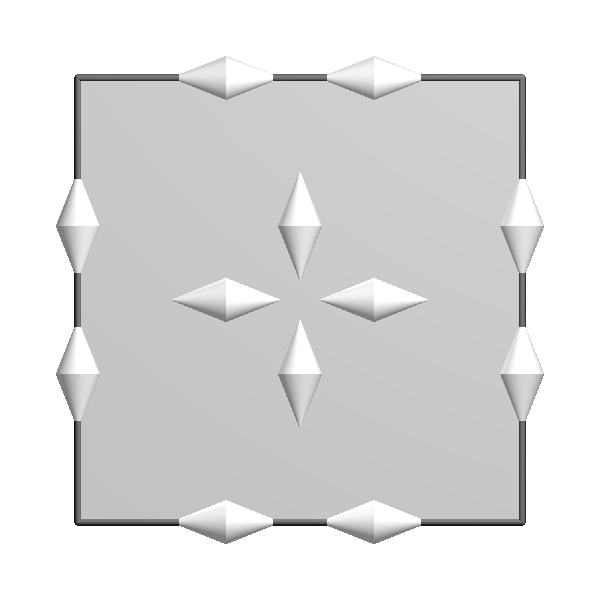} &
\includegraphics[width=30mm]{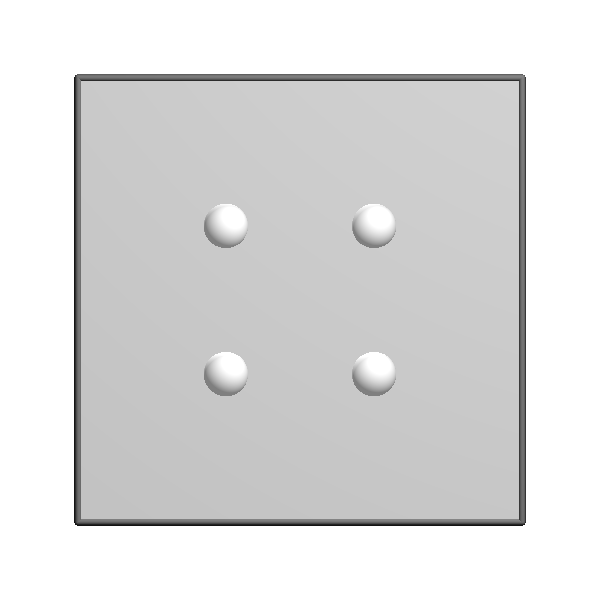}
\\
$\Q_2\x \Q_2 $& $\P_1$ & $\RT_2$ & $\Q_1$
\\
\multicolumn{2}{c}{$\underbrace{\hspace{55mm}}_{\text{\normalsize Stokes}}$} &
\multicolumn{2}{c}{$\underbrace{\hspace{55mm}}_{\text{\normalsize mixed Poisson}}$}
\end{tabular}}
\caption{Degrees of freedom and transformations
used to construct the first elements.}\label{dofs1}
\end{figure}

The Stokes pair $W_h$, $Q_h$ is a standard Stokes element, the $Q_2$--$P_1$
element, for which the inf-sup condition (S3) is well known.
See \cite[Chapter~II, \S3.2]{Girault-Raviart}.
The mixed Poisson pair $S_h$, $U_h$ is a standard choice as well, the quadrilateral
Raviart--Thomas elements of second lowest order.  A proof of the inf-sup condition (S4)
for general quadrilateral meshes is given, e.g., in \cite{ABF2005}.  We have thus verified the hypotheses of Theorem~\ref{s2}.
Therefore if we define $\Sigma_h=S_h\x S_h$ and $V_h=U_h\x U_h$, the triple
$\Sigma_h$, $V_h$, $Q_h$ satisfies (S2).

From the definitions of $S_h$ and $U_h$, it follows that \eqref{vdef} holds with $\hat V = \hat U\x \hat U$
and \eqref{sigdef} holds with $\hat \Sigma=\hat S\x\hat S$.  Since $\div \hat S\subset \hat U$,
\eqref{incl} holds.  Theorem~\ref{s1} thus applies, showing that the spaces
 $\Sigma_h$, $V_h$, $Q_h$ satisfy (S1) as well.  Thus we have indeed constructed
a stable triple of spaces for the elasticity problem, satisfying the stability
condition (S0) and therefore the quasioptimality estimate \eqref{quasiopt}.
The diagram for the elements are shown in Figure~\ref{elts1}.
\begin{figure}[htb]
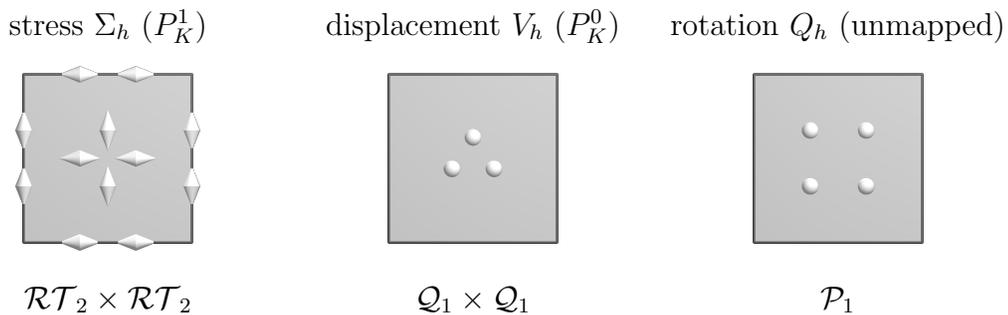

\centerline{%
\begin{tabular}{>{\centering\arraybackslash}p{45mm}>{\centering\arraybackslash}p{45mm}>{\centering\arraybackslash}p{45mm}}
 stress $\Sigma_h$ ($P_K^1$) & displacement $V_h$ ($P_K^0$) & rotation $Q_h$ (unmapped)
\\
\includegraphics[width=30mm]{rtsquare2d2.png} &
\includegraphics[width=30mm]{p2d1.png} &
 \includegraphics[width=30mm]{dgsquare2d1.png}
\\
$\RT_2\x\RT_2$ & $\Q_1\x \Q_1 $& $\P_1$ 
\end{tabular}}
\caption{The first choice of elasticity elements.}\label{elts1}
\end{figure}

\subsection{Higher order elements}
\label{subsection_space2}
The above elements generalize directly to arbitrary order $r\ge 2$.  For the Stokes element
we use $\Q_r$-$\P_{r-1}$, and for the mixed Poisson element we use
$\RT_r$-$\Q_{r-1}$.

\subsection{A simpler element}
\label{subsection_space3}
In this section we derive a simpler element.  The stress is approximated by
the lowest order quadrilateral BDM elements, which is constructed from
an $8$-dimensional space $\BDM_1$ of reference shape functions, spanned by $\P_1$
vector fields together with the two vector fields $\curl \hat x_1^2 x_2$
$\curl\hat x_1 x_2^2$.   The displacement and rotation
spaces simply consist of piecewise constants.  This element is thus a quadrilateral
analogue of the simple triangular finite element for elasticity with weak
symmetry introduced in \cite{AFW2006} and \cite{AFW2007}.
The elasticity element is summarized in Figure~\ref{elts2}.
\begin{figure}[htb]
\centerline{%
\begin{tabular}{>{\centering\arraybackslash}p{45mm}>{\centering\arraybackslash}p{45mm}>{\centering\arraybackslash}p{45mm}}
 stress $\Sigma_h$ ($P_K^1$) & displacement $V_h$ & rotation $Q_h$
\\
\includegraphics[width=30mm]{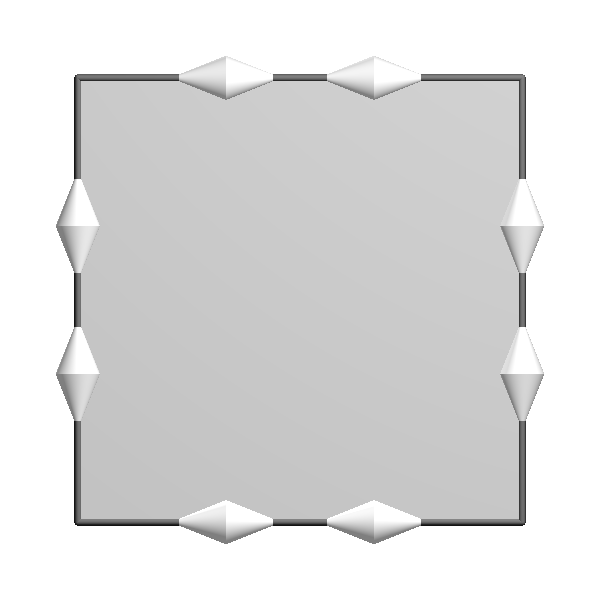} &
\includegraphics[width=30mm]{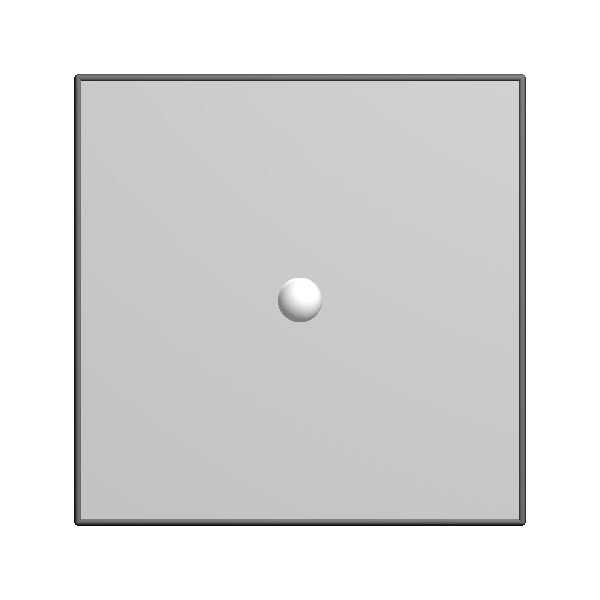} &
\includegraphics[width=30mm]{X2square2d0.png}
\\
$\BDM_1\x\BDM_1$ & $\P_0\x\P_0$& $\P_0$ 
\end{tabular}}
\caption{A simple stable choice of elasticity elements.}\label{elts2}
\end{figure}

Note that the mixed Poisson gradient space is based on $\BDM_1$
rather than $\RT_2$ as in the first element.
For analysis, we define the Stokes velocity space using the serendipity space $\SS_2$
instead of $\Q_2$.  The space of serendipity polynomials $\SS_r$ is defined
to be the span of $\P_r$ and the two polynomials $x_1^r x_2$ and $x_2 x_2^r$,
and the space $\BDM_r$ is the span of $\P_r\x\P_r$ and the two vector fields
$\curl x_1^{r+1} x_2$ and $\curl x_1 x_2^{r+1}$.
Thus, for this element, the reference shape functions are 
$$
\hat W = \SS_2(\hat K)\x\SS_2(\hat K), \quad
\hat S = \BDM_1(\hat K), \quad \hat U=\P_0(\hat K),
$$
and the spaces $W_h$, $S_h$ and $U_h$ are then defined by \eqref{defwh}, \eqref{defsh},
\eqref{defuh}.  Note that the crucial compatibility condition
$\curl\hat W\subset \hat S\x\hat S$ again holds.  The remaining space is
\begin{equation*}
 Q_h = \{\, q\in L^2(\Omega)\,:\, q|_K\in \P_0(K),\ \forall K\in\T_h\,\}.
\end{equation*}
Since constants on the reference element map by $P_K^0$ to constants on the
element $K$, for this element $Q_h$ and $U_h$ coincide, and are simply the space of
piecewise constant functions.  The element diagrams for these auxiliary spaces
are given in Figure~\ref{dofs2}.
\begin{figure}[htb]
\centerline{%
\begin{tabular}{cccc}
 \includegraphics[width=30mm]{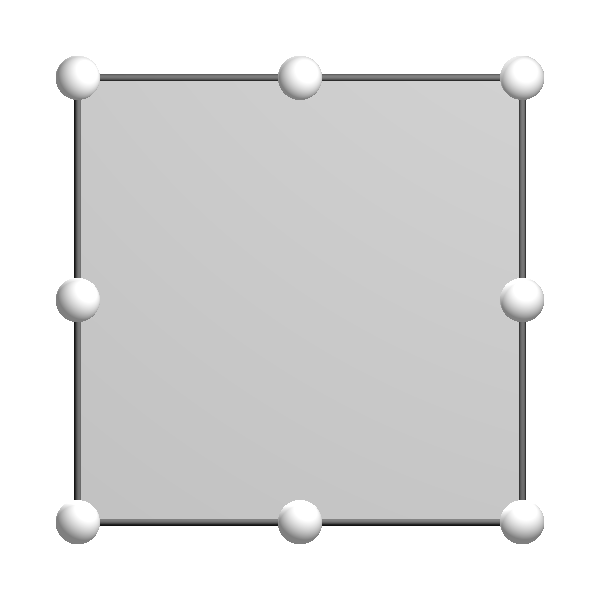} &
\includegraphics[width=30mm]{X2square2d0.png} &
\includegraphics[width=30mm]{X1square2d1.png} &
\includegraphics[width=30mm]{X2square2d0.png}
\\
$\SS_2\x\SS_2 $& $\P_0$ & $\BDM_1$ & $\P_0$
\\
\multicolumn{2}{c}{$\underbrace{\hspace{55mm}}_{\text{\normalsize Stokes}}$} &
\multicolumn{2}{c}{$\underbrace{\hspace{55mm}}_{\text{\normalsize mixed Poisson}}$}
\end{tabular}}
\caption{Degrees of freedom used to construct the second elements.}\label{dofs2}
\end{figure}
This Stokes element is the one referred to as $Q^{(8)}_2$--$P_0$ in \cite{Fortin81},
for which it is easy to prove stability using
the edge degrees of freedom.  This is discussed in \cite{Fortin81}, where
it is shown  the inf-sup condition (S3) holds (on general quadrilateral
meshes) for a variant of the
element ($R^{(8)}_2$--$P_0$) which uses the same pressure space and a smaller
velocity space.  This of course implies the inf--sup condition with
the larger velocity space.  The $BDM_1$--$P_0$ element is a standard
stable mixed finite element for the Poisson equation.  Its stability
on general quadrilateral meshes is shown, for instance, in \cite{ABF2005}.
Thus all the hypotheses of Theorems~\ref{s2} and \ref{s1} are again met,
and the choice $\Sigma_h=S_h\x S_h$, $V_h=U_h\x U_h$, and $Q_h$ give
a stable element for elasticity.

\section{$L^2$ estimates and rates of convergence}\label{s:rates}

The rate of convergence that can be deduced from the quasioptimal error estimate
\eqref{quasiopt} is limited by the approximation properties of the finite element space
$\Sigma_h$ in the $H(\div)$ norm.  We can demonstrate higher rates of convergence by
establishing a bound in the $L^2$ norm, as we do in this section.

In order to obtain an estimate in $L^2(\Omega,\Mat) \x L^2(\Omega,\mathbb{R}^2)\x L^2(\Omega)$,
we impose a further condition:
\begin{itemize}
 \item [(S5)]  There exists a projection $\Pi_h$ from $H^{1}(\Omega, \mathbb{M})$ 
onto $\Sigma_h$ such that
 $$P_{V_h}\div\Pi_h\sigma=P_{V_h}\div\sigma.$$
\end{itemize}
Here $P_{V_h}:L^2(\Omega,\mathbb{R}^2)\to V_h$ is the $L^2$-projection.

\begin{thm}\label{l2conv}
 Suppose that conditions {\rm(S0)} and {\rm(S5)} are satified.  Then
\begin{multline*}
\norm{\sigma-\sigma_h}_{L^2}+\norm{u-u_h}_{L^2}+\norm{p-p_h}_{L^2}
\\ \le C(
\norm{\sigma-\Pi_h\sigma}_{L^2}+\norm{u-P_{V_h}u}_{L^2}+\norm{p-P_{Q_h}p}_{L^2}).
\end{multline*}
\end{thm}
\begin{proof}
We decompose the error into the projected error
$$
\eta_h=(\Pi_h\sigma-\sigma_h,P_{V_h}u-u_h,P_{Q_h}p-p_h)\in Y_h,
$$
and the projection error
$$
\bar\eta_h =  (\sigma-\Pi_h\sigma,u-P_{V_h}u,p-P_{Q_h}p).
$$
Making use
of the triangle inequality, it suffices to show that
$\norm{\eta_h}_{L^2} \le C \norm{\bar\eta_h}$.

By the inf-sup condition (S0), there
exists a non-zero $z=(\tau,v,q)\in Y_h$ such that
\begin{equation}\label{eta1}
B(\eta_h,z)\ge c_0\norm{\eta_h}_Y\norm{z}_Y \ge c_0\norm{\eta_h}_{L^2}\norm{z}_Y.
\end{equation}
Now, by Galerkin orthogonality,
\begin{equation}\label{eta2}
B(\eta_h,z)=-B(\bar\eta_h,z).
\end{equation}
The quantity $B(\bar\eta_h,z)$ is a sum of five terms according to
the definition \eqref{defB} of the bilinear form, but the fourth term,
$(\div(\sigma-\Pi_h\sigma),v)$, 
vanishes, because of the assumption (S5).  We then have
\begin{equation}\label{eta3}
B(\bar\eta_h,z)\le C\norm{\bar\eta_h}_{L^2}\norm{z}_Y,
\end{equation}
where $C$ depends only on an upper bound for $A$. Combining
\eqref{eta1}, \eqref{eta2}, and \eqref{eta3},
we conclude that $\norm{\eta_h}_{L^2}\le c_0^{-1}C\norm{\bar\eta_h}_{L^2}$.
\qed
\end{proof}

We now give a simple criteria which makes it easy to verify that 
all finite element spaces introduced in Section~\ref{s:elts} satisfy assumption (S5). See \cite{ABF2005} for details on the verification. 

\begin{lemma}
\label{verification_s5}
Let $\hat{\Pi}$ be a bounded projection operator from $H^{1}(\hat{K},\mathbb{M})$ onto $\hat{\Sigma}$ such that 
\begin{equation}
\label{req_s5}
\div \hat{\Pi}\hat{\sigma} = P_{\hat{V}}\div \hat{\sigma},\quad \forall \hat{\sigma}\in H^{1}(\hat{K},\mathbb{M}),
\end{equation}
where $P_{\hat{V}}$ is the $L^{2}$-projection onto $\hat{V}$. 
Define $\Pi_{h}:H^{1}(\Omega,\mathbb{M})\to V_{h}$ by
\begin{equation*}
\Pi_{h}\sigma |_{K} = P_{K}^{1}\hat\Pi(P_{K}^{1})^{-1}(\sigma|_K),\quad\forall K\in \T_{h}.
\end{equation*} 
Then, we have that 
\begin{equation*}
P_{V_h}\div\Pi_h\sigma=P_{V_h}\div\sigma,\quad \forall \sigma \in H^{1}(\Omega,\mathbb{M}).
\end{equation*}
\end{lemma}

\begin{proof}
Given $\sigma\in H^{1}(\Omega,\mathbb{M})$, we have $\hat\sigma:=(P_{K}^{1})^{-1}\left(\sigma |_{K}\right) \in H^{1}(\hat{K},\mathbb{M})$ 
for any $K\in \mathcal{T}_{h}$. For  $v\in V_{h}$, we have $\hat v:=\left( P_{K}^{0}\right)^{-1}\left( v |_{K}\right)\in \hat{V}$ and by \eqref{transprop}, \eqref{transint} and \eqref{req_s5}, we have
\begin{align*}
 (\div \Pi_{h}\sigma, v)_{L^{2}(K)}
&=  (\div P_{K}^{1}\hat{\Pi}\hat\sigma, P_K^0 \hat v)_{L^{2}(K)}
=  (P_{K}^{2}\div \hat{\Pi}\hat \sigma,P_K^0 \hat v)_{L^{2}(K)} 
=  (\div \hat{\Pi}\hat\sigma, \hat v)_{L^{2}(\hat{K})}  
\\
&=  (P_{\hat{V}}\div \hat\sigma, \hat v )_{L^{2}(\hat{K})} 
=  (\div \hat\sigma, \hat v )_{L^{2}(\hat{K})}  
=  (\div\sigma, v)_{L^{2}(K)}.
\end{align*}
\qed
\end{proof}

\subsection{Approximation properties on quadrilateral meshes}

We now recall some results on the approximation rates achieved
by finite element spaces on shape regular meshes of convex
quadrilaterals.  In \cite{ABF2002} it is shown that if $X_h$
is a finite element space of scalar functions derived from
shape function spaces $X_K$ which are themselves obtained from a
reference shape function space $\hat X$ via the transformation
$P_K^0$, then $X_h$ achieves approximation order $r+1$ in
the $L^2$ norm if and only if $\Q_{r} \subset \hat X$.
In \cite{ABF2005}, it shown that if $X_h$ is a finite element space
of vector fields derived from shape function spaces $X_K$ defined
from a reference space $\hat X$ via the \emph{Piola transform}
$P_K^1$, then a necessary and sufficient condition for order $r+1$
approximation in the $L^2$ norm is that $\mathcal{U}_r \subset \hat X$ while
the condition for order $r+1$ approximation of $\div u$ in the $L^2$
is $\mathcal{R}_r \subset \div \hat X$.
Here $\mathcal{U}_r$ is the subspace of codimension $1$ of
$\RT_{r+1}$ defined as the span of the vector fields
\begin{equation*}
(\hat{x}_1^i \hat{x}_2^j , 0),\    (0, \hat{x}_1^j \hat{x}_2^i),\quad   0 \leq i \leq r+1,\ 0 \leq j \leq r,
\end{equation*}
except that the two vector fields $( \hat{x}_1^{r+1} \hat{x}_2^r , 0)$
and $(0,\hat{x}_1^r \hat{x}_2^{r+1})$ are replaced by 
the single vector field $(\hat{x}_1^{r+1} \hat{x}_1^r , -\hat{x}_1^r \hat{x}_1^{r+1})$.
The space $\mathcal{R}_r$ is the subspace of codimension $1$ of $\Q_{r+1}$ spanned by
all its monomials except $\hat{x}_1^{r+1}\hat{x}_2^{r+1}$.

\subsection{Rates of convergence of the proposed elements}

Our first choice of finite element spaces is built from the reference space $\RT_2$ transformed by $P_K^1$,
the space $Q_1$ transformed by $\P_K^0$, and the space $P_1$, not subject to a transformation,
as depicted in Figure~\ref{elts1}.
It follows that each of these spaces achieves quadratic convergence in $L^2$.  In light of
Theorem~\ref{l2conv}, the finite element solution converges quadratically in $L^2$ for all variables if
the solution is smooth.  Concerning approximation of the divergence, we have $\mathcal R_0\subset Q_1=\div\RT_2$,
but $\mathcal R_1\nsubseteq\div\RT_2$, so the approximation error in $H(\div)$ is only first order (and
so the finite element method converges with first order in $H(\div)$ by \eqref{quasiopt}.
Similarly, the higher order methods of this family, described in
Section~\ref{subsection_space2}, achieve order $r$ convergence in $L^2$ for all variables, but in $H(\div)$
the convergence order for the stress is reduced to $r-1$.  Of course on meshes in which all the elements
are square, or, more generally, parallelograms, the rate of convergence in $H(\div)$
is $r$.

Similar reasoning, applied to the simple choice of elements described in Section~\ref{subsection_space3} and
illustrated in Figure~\ref{elts2}, establishes linear convergence for all variables in $L^2$.
However, since $\mathcal R_0\nsubseteq \P_1=\div\BDM_1$,
we do not expect any convergence in $H(\div)$ on general quadrilateral meshes.

\section{Numerical results}\label{s:numer}

In this section, we present simple numerical results which illustrates the error estimates just obtained.
We take the domain to be the unit square and consider two sequences of meshes, the first using uniform meshes
into subsquares, and the second consisting of meshes in which every element is congruent to a fixed
trapezoid, as illustrated in Figure~\ref{meshes}.  The trapezoidal mesh sequence was
introduced in \cite{ABF2002} to study finite element approximation on quadrilateral meshes.   
For the test problem we take the elasticity system with homogeneous Dirichlet boundary conditions
and the exact solution
\begin{align*}
u_1 = \cos (\pi x) \sin (2\pi y),\quad
u_2 = \sin (\pi x) \cos( \pi y).
\end{align*}
The body force $f$ is then determined using the values
$\lambda = 123$ and $\mu = 79.3$ for the Lam\'e coefficients.

\begin{figure}[htb]
\centerline{%
\includegraphics[width=40mm]{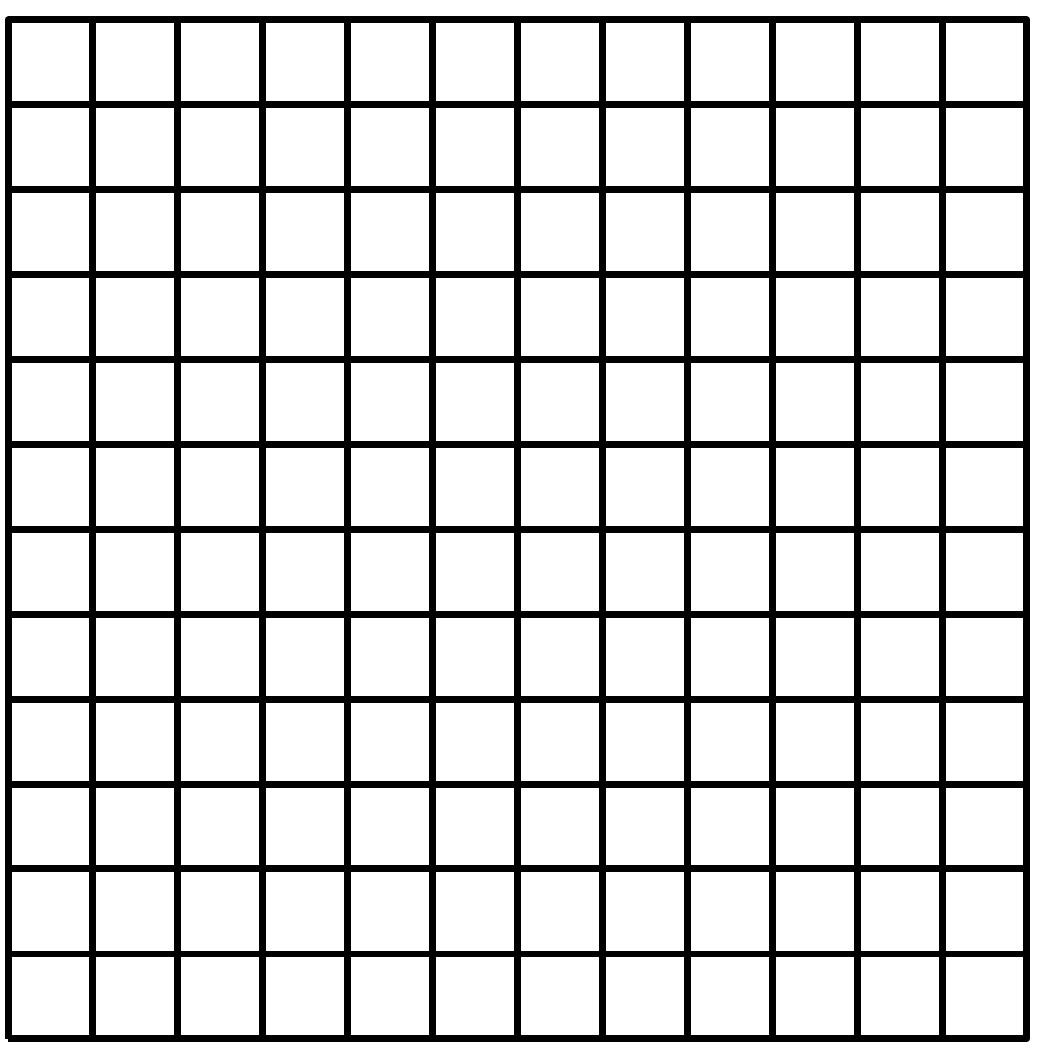}\quad
\includegraphics[width=40mm]{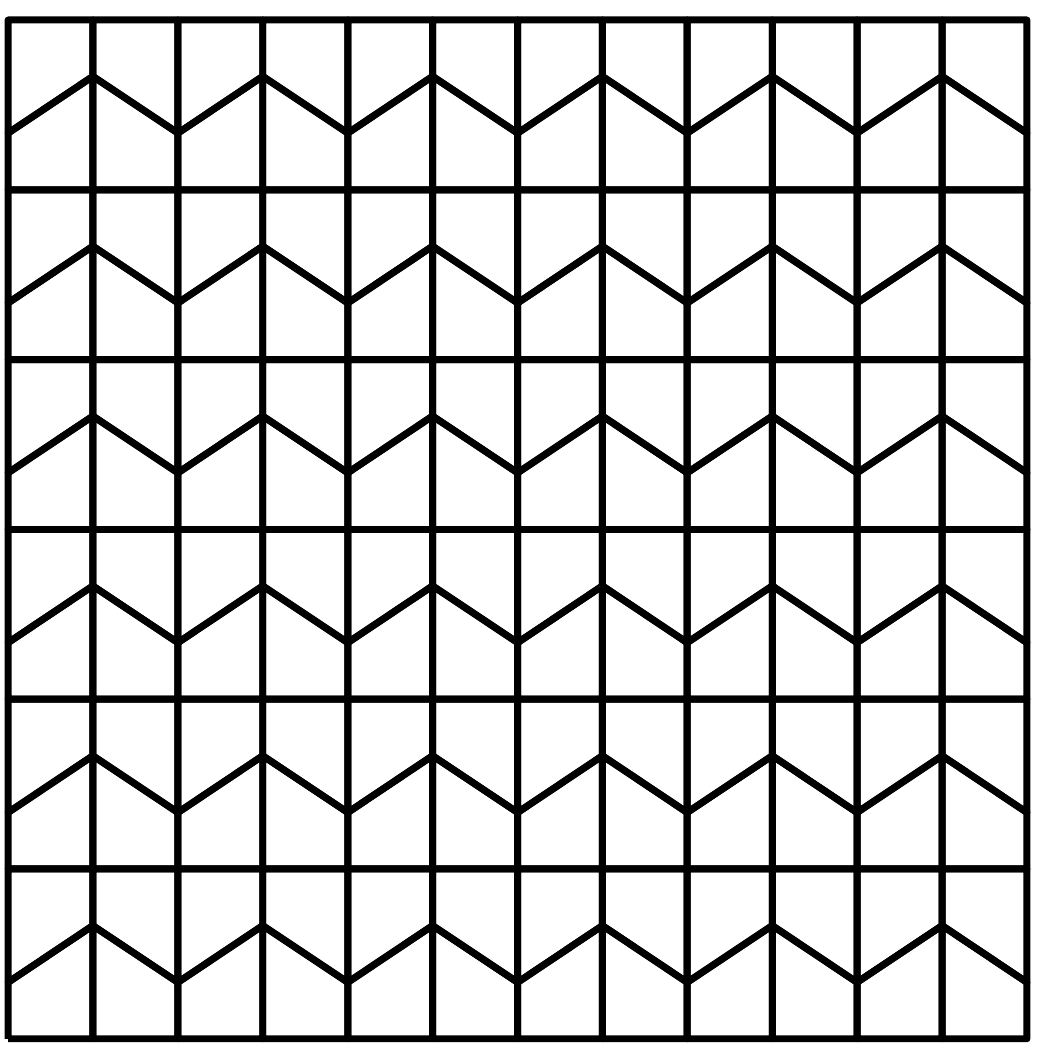}
}
\caption{Square and trapezoidal meshes.}\label{meshes}
\end{figure}

In Table~\ref{numer_tab1},  we show errors and convergence rates in the $L^{2}$ norm for $\sigma$, $\div \sigma$, $u$ and $p$,
using the elements of Section~\ref{subsection_space1}.  As expected all three variables converge quadratically
in $L^2$, while $\div\sigma$ converges only linearly with trapezoidal meshes, and quadratically for square meshes.
Table~\ref{numer_tab2} illustrates the same quantities for the simple 
stable choice of elasticity elements of Section~\ref{subsection_space3}, showing the expected linear convergence,
which reduces to no convergence for the divergence computed with trapezoidal meshes.
 
\begin{table}[tp]
\caption{Convergence results for the elements of Section~\ref{subsection_space1} (illustrated in Figure~\ref{elts1}).} \label{numer_tab1}
\centering
\begin{tabular}{c|clc|clc}
\multicolumn{7}{c}{\rule{0pt}{15pt}\emph{Square meshes}}\\[5pt]
\hline
\rule{0pt}{15pt}
& \multicolumn{3}{c|}{$\Vert \sigma - \sigma_{h} \Vert_{L^{2}(\Omega)}$}  & 
  \multicolumn{3}{c}{$\Vert \div (\sigma-\sigma_h)\Vert_{L^2(\Omega)}$}  \\[5pt]
$h$ & error & $\quad\%$ & order  & error & $\quad\%$ & order\\
\hline
\rule{0pt}{15pt}%
$1/2$ & $3.06$e$+$$2$ & $31.8$ & & $1.83$e$+$$3$ & $35.9$ & \\
$1/4$ & $6.64$e$+$$1$ & $\06.91$ & $2.2$ & $4.19$e$+$$2$ & $\08.21$ & $2.1$\\
$1/8$ & $1.59$e$+$$1$ & $\01.65$ & $2.1$ & $1.07$e$+$$2$ & $\02.10$  & $2.0$\\
$1/16$ & $3.88$e$+$$0$ & $\00.403$ & $2.0$ & $2.70$e$+$$1$ & $\00.529$ & $2.0$\\
$1/32$ & $9.61$e$-$$1$ & $\00.0998$ & $2.0$ & $6.77$e$+$$0$ & $\00.132$ & $2.0$\\
$1/64$ & $2.39$e$-$$1$ & $\00.0248$ & $2.0$ & $1.69$e$+$$0$ & $\00.0331$ & $2.0$\\
$1/128$ & $5.98$e$-$$2$ & $\00.00621$ & $2.0$ & $4.23$e$-$$1$ & $\00.00828$ & $2.0$\\
\hline
\rule{0pt}{15pt}
 & \multicolumn{3}{c|}{$\Vert u-u_h\Vert_{L^{2}(\Omega)}$}  & \multicolumn{3}{c}{$\Vert p-p_h\Vert_{L^{2}(\Omega)}$} \\[5pt]
$h$ & error & $\quad\%$ & order  & error & $\quad\%$ & order\\
\hline
\rule{0pt}{15pt}%
$1/2$ & $2.33$e$-$$1$ & $33.0$ & & $7.28$e$-$$1$ & $41.5$ &\\
$1/4$ & $4.87$e$-$$2$ & $\06.89$ & $2.3$ & $2.17$e$-$$1$ & $12.4$ & $1.7$\\
$1/8$ & $1.24$e$-$$2$ & $\01.76$ & $2.0$ & $5.60$e$-$$2$ & $\03.19$ & $2.0$\\
$1/16$ & $3.12$e$-$$3$ & $\00.442$ & $2.0$ & $1.40$e$-$$2$ & $\00.800$ & $2.0$\\
$1/32$ & $7.82$e$-$$4$ & $\00.110$ & $2.0$ & $3.51$e$-$$3$  & $\00.200$ & $2.0$\\
$1/64$ & $1.95$e$-$$4$ & $\00.0276$ & $2.0$ & $8.78$e$-$$4$ & $\00.0500$ & $2.0$\\
$1/128$ & $4.89$e$-$$5$ & $\00.00691$ & $2.0$ & $2.19$e$-$$4$ & $\00.0125$ & $2.0$\\
\hline
\multicolumn{7}{c}{\rule{0pt}{15pt}\emph{\rule{0pt}{20pt}Trapezoidal meshes}}\\[5pt]
\hline
\rule{0pt}{15pt}
& \multicolumn{3}{c|}{$\Vert \sigma - \sigma_{h} \Vert_{L^{2}(\Omega)}$}  & 
  \multicolumn{3}{c}{$\Vert \div (\sigma-\sigma_h)\Vert_{L^2(\Omega)}$}  \\[5pt]
$h$ & error & $\quad\%$ & order  & error & $\quad\%$ & order\\
\hline
\rule{0pt}{15pt}%
$1/2$ & $3.35$e$+$$2$ & $34.8$ & & $2.06$e$+$$3$ & $40.4$ & \\
$1/4$ & $8.93$e$+$$1$ & $\09.29$ & $1.9$ & $5.95$e$+$$2$ & $11.6$ & $1.8$\\
$1/8$ & $2.11$e$+$$1$ & $\02.19$ & $2.0$ & $1.84$e$+$$2$ & $\03.60$  & $1.6$\\
$1/16$ & $5.24$e$+$$0$ & $\00.560$ & $2.0$ & $7.14$e$+$$1$ & $\01.40$ & $1.3$\\
$1/32$ & $1.30$e$+$$0$ & $\00.135$ & $2.0$ & $3.25$e$+$$1$ & $\00.636$ & $1.1$\\
$1/64$ & $3.26$e$-$$1$ & $\00.0339$ & $2.0$ & $1.58$e$+$$1$ & $\00.310$ & $1.0$\\
$1/128$ & $8.16$e$-$$2$ & $\00.00847$ & $2.0$ & $7.87$e$+$$0$ & $\00.154$ & $1.0$\\
\hline
\rule{0pt}{15pt}
 & \multicolumn{3}{c|}{$\Vert u-u_h\Vert_{L^{2}(\Omega)}$}  & \multicolumn{3}{c}{$\Vert p-p_h\Vert_{L^{2}(\Omega)}$} \\[5pt]
$h$ & error & $\quad\%$ & order  & error & $\quad\%$ & order\\
\hline
\rule{0pt}{15pt}%
$1/2$ & $2.59$e$-$$1$ & $36.7$ & & $7.34$e$-$$1$ & $41.8$ &\\
$1/4$ & $6.57$e$-$$2$ & $\09.30$ & $1.9$  & $2.61$e$-$$1$ & $14.9$ & $1.4$\\
$1/8$ & $1.57$e$-$$2$ & $\02.23$ & $2.0$ & $7.12$e$-$$2$ & $\04.06$ & $1.8$\\
$1/16$ & $3.96$e$-$$3$ & $\00.560$ & $1.9$ & $1.79$e$-$$2$ & $\01.02$ & $1.9$\\
$1/32$ & $9.93$e$-$$4$ & $\00.140$ & $2.0$ & $4.49$e$-$$3$  & $\00.256$ & $2.0$\\
$1/64$ & $2.48$e$-$$4$ & $\00.0351$ & $2.0$ & $1.12$e$-$$3$  & $\00.064$ & $2.0$\\
$1/128$ & $6.20$e$-$$5$ & $\00.00878$ & $2.0$ & $2.81$e$-$$4$ & $\00.0160$ & $2.0$\\
\hline
\end{tabular}
\end{table}

\begin{table}[tp]
\caption{Convergence results for the elements of Section~\ref{subsection_space3} (illustrated in Figure~\ref{elts2}).} \label{numer_tab2}
\centering
\begin{tabular}{c|clc|clc}
\multicolumn{7}{c}{\rule{0pt}{15pt}\emph{Square meshes}}\\[5pt]
\hline
\rule{0pt}{15pt}
& \multicolumn{3}{c|}{$\Vert \sigma - \sigma_{h} \Vert_{L^{2}(\Omega)}$}  & 
  \multicolumn{3}{c}{$\Vert \div (\sigma-\sigma_h)\Vert_{L^2(\Omega)}$}  \\[5pt]
$h$ & error & $\quad\%$ & order  & error & $\quad\%$ & order\\
\hline
\rule{0pt}{15pt}%
$1/2$ & $6.20$e$+$$2$ & $64.5$ & & $3.40$e$+$$3$ & $66.5$ & \\
$1/4$ & $2.51$e$+$$2$ & $26.2$ & $1.3$ & $2.28$e$+$$3$ & $44.8$ & $0.5$\\
$1/8$ & $1.09$e$+$$2$ & $11.4$ & $1.2$ & $1.18$e$+$$3$ & $23.3$  & $0.9$\\
$1/16$ & $5.23$e$+$$1$ & $\05.43$ & $1.1$ & $6.00$e$+$$2$ & $11.7$ & $1.0$\\
$1/32$ & $2.58$e$+$$1$ & $\02.68$ & $1.0$ & $3.01$e$+$$2$ & $\05.89$ & $1.0$\\
$1/64$ & $1.28$e$+$$1$ & $\01.34$ & $1.0$ & $1.50$e$+$$2$ & $\02.95$ & $1.0$\\
$1/128$ & $6.42$e$+$$0$ & $\00.667$ & $1.0$ & $7.53$e$+$$1$ & $\01.47$ & $1.0$\\
\hline
\rule{0pt}{15pt}
 & \multicolumn{3}{c|}{$\Vert u-u_h\Vert_{L^{2}(\Omega)}$}  & \multicolumn{3}{c}{$\Vert p-p_h\Vert_{L^{2}(\Omega)}$} \\[5pt]
$h$ & error & $\quad\%$ & order  & error & $\quad\%$ & order\\
\hline
\rule{0pt}{15pt}%
$1/2$ & $4.29$e$-$$1$ & $60.7$ & & $1.63$e$+$$0$ & $93.4$ &\\
$1/4$ & $2.90$e$-$$1$ & $41.1$ & $0.5$ & $7.97$e$-$$1$ & $45.4$ & $1.0$\\
$1/8$ & $1.49$e$-$$1$ & $21.1$ & $1.0$ & $4.13$e$-$$1$ & $23.6$ & $0.9$\\
$1/16$ & $7.48$e$-$$2$ & $10.6$ & $1.0$ & $2.08$e$-$$1$ & $11.9$ & $1.0$\\
$1/32$ & $3.74$e$-$$2$ & $\05.30$ & $1.0$ & $1.04$e$-$$1$ & $\05.94$ & $1.0$\\
$1/64$ & $1.87$e$-$$2$ & $\02.65$ & $1.0$ & $5.21$e$-$$2$ & $\02.97$ & $1.0$\\
$1/128$ & $9.37$e$-$$3$ & $\01.32$ & $1.0$ & $2.61$e$-$$2$ & $\01.49$ & $1.0$\\
\hline
\multicolumn{7}{c}{\rule{0pt}{15pt}\emph{\rule{0pt}{20pt}Trapezoidal meshes}}\\[5pt]
\hline
\rule{0pt}{15pt}
& \multicolumn{3}{c|}{$\Vert \sigma - \sigma_{h} \Vert_{L^{2}(\Omega)}$}  & 
  \multicolumn{3}{c}{$\Vert \div (\sigma-\sigma_h)\Vert_{L^2(\Omega)}$}  \\[5pt]
$h$ & error & $\quad\%$ & order  & error & $\quad\%$ & order\\
\hline
\rule{0pt}{15pt}%
$1/2$ & $6.67$e$+$$2$ & $69.3$ & & $3.70$e$+$$3$ & $72.4$ & \\
$1/4$ & $2.90$e$+$$2$ & $30.2$ & $1.1$ & $2.58$e$+$$3$ & $50.6$ & $0.52$\\
$1/8$ & $1.22$e$+$$2$ & $12.7$ & $1.2$ & $1.59$e$+$$3$ & $31.3$  & $0.69$\\
$1/16$ & $5.77$e$+$$1$ & $\06.00$ & $1.0$ & $1.19$e$+$$3$ & $23.4$ & $0.42$\\
$1/32$ & $2.84$e$+$$1$ & $\02.95$ & $1.0$ & $1.06$e$+$$3$ & $20.8$ & $0.16$\\
$1/64$ & $1.41$e$+$$1$ & $\01.46$ & $1.0$ & $1.03$e$+$$3$ & $20.2$ & $0.05$\\
$1/128$ & $7.03$e$+$$0$ & $\00.731$ & $1.0$ & $1.02$e$+$$3$ & $20.0$ & $0.01$\\
\hline
\rule{0pt}{15pt}
 & \multicolumn{3}{c|}{$\Vert u-u_h\Vert_{L^{2}(\Omega)}$}  & \multicolumn{3}{c}{$\Vert p-p_h\Vert_{L^{2}(\Omega)}$} \\[5pt]
$h$ & error & $\quad\%$ & order  & error & $\quad\%$ & order\\
\hline
\rule{0pt}{15pt}%
$1/2$ & $4.72$e$-$$1$ & $66.8$ & & $1.70$e$+$$0$ & $97.0$ &\\
$1/4$ & $2.97$e$-$$1$ & $42.1$ & $0.6$ & $9.05$e$-$$1$ & $51.6$ & $0.9$\\
$1/8$ & $1.60$e$-$$1$ & $22.6$ & $0.8$ & $4.46$e$-$$1$ & $25.4$ & $1.0$\\
$1/16$ & $8.05$e$-$$2$ & $11.4$ & $1.0$ & $2.25$e$-$$1$ & $12.8$ & $0.9$\\
$1/32$ & $4.03$e$-$$2$ & $\05.70$ & $1.0$ & $1.12$e$-$$1$ & $\06.42$ & $1.0$\\
$1/64$ & $2.01$e$-$$2$ & $\02.85$ & $1.0$ & $5.64$e$-$$2$ & $\03.21$ & $1.0$\\
$1/128$ & $1.00$e$-$$2$ & $\01.43$ & $1.0$ & $2.82$e$-$$2$ & $\01.61$ & $1.0$\\
\hline
\end{tabular}
\end{table}

In Figure~\ref{locking_free}, we show numerical evidence of the locking-free property of 
the BDM type elements of Section~\ref{subsection_space3} (illustrated in Figure~\ref{elts2}) on trapezoidal meshes. 
The exact solution is the same as above and the Young's modulus $E$ is taken as $1000$.
The two figures show the convergence history of the stress and displacement as a function of the total number of degrees of freedom for the stress, the displacement and the rotation. We used various values of the Poisson ratio $\nu$ close to the limiting value of $0.5$. Recall that 
$$
\lambda=\frac{E \nu}{(1+\nu)(1-2 \nu)}.
$$

\begin{figure}[htb]
\centerline{\includegraphics[width=.5\textwidth]{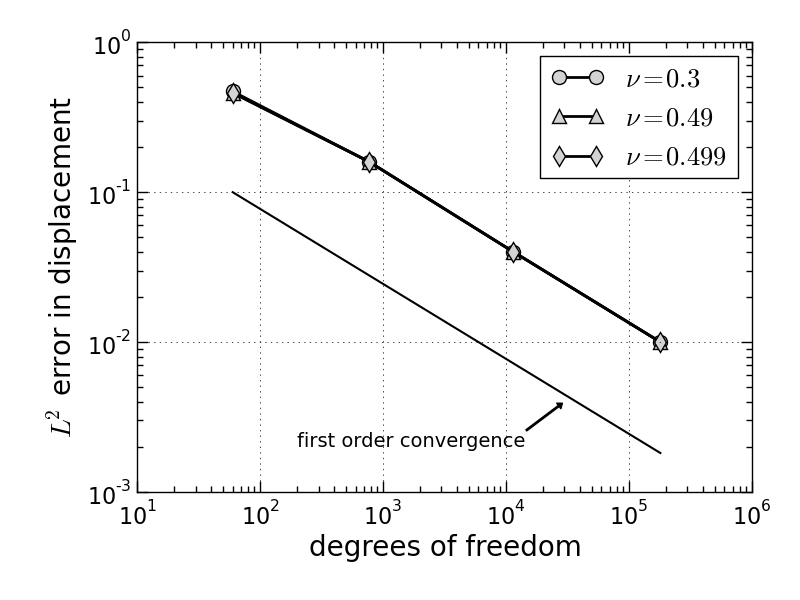}
\includegraphics[width=.5\textwidth]{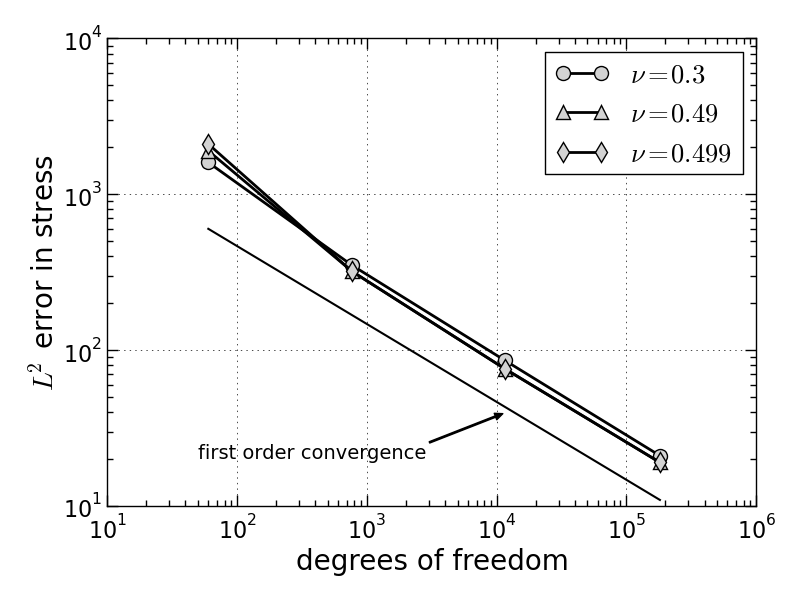}}
\caption{Error of the first order method for several values of the Poisson ratio,
displacement on left, stress on the right.
The curves nearly coincide, illustrating the absence of locking.}
\label{locking_free}
\end{figure} 

\bibliographystyle{amsplain}
\bibliography{quadelas}

\end{document}